\documentclass[12pt, notitlepage]{amsart}
\usepackage{latexsym, amsfonts, amsmath, amssymb, amsthm, cite}
\usepackage{verbatim}

\newtheorem{lemma}{Lemma}
\newtheorem{proposition}{Proposition}
\newtheorem{theorem}{Theorem}

\usepackage{graphicx}
\usepackage{mathrsfs}

\begin{document}
   
\title{Selberg's central limit theorem for $\log |\zeta(\tfrac 12+it)|$}
\  \author{Maksym Radziwi\l\l \   and K. Soundararajan} 
 \address{Department of Mathematics \\ Rutgers University \\ 110 Frelinghuysen Rd. \\ Piscataway \\  NJ 08854-8019} 
 \email{maksym.radziwill@gmail.com}
\address{Department of Mathematics \\ Stanford University \\
450 Serra Mall, Bldg. 380\\ Stanford \\ CA 94305-2125}
\email{ksound@math.stanford.edu}
\thanks{The first author was partially supported by NSF grant DMS-1128155. The second author is partially supported by NSF grant DMS-1500237, and a Simons Investigator award from the Simons Foundation} 

\date{\today}
\maketitle

\section{Introduction} 
\noindent
In this paper we give a new and simple proof of Selberg's influential theorem \cite{Selberg1, Selberg2} that $\log |\zeta(\tfrac 12+it)|$ has an approximately normal distribution with mean zero and variance $\tfrac 12 \log \log |t|$.  Apart from some basic facts about the Riemann zeta function, we have tried to make our proof  self-contained. 

\begin{theorem} \label{mainthm} Let $V$ be a fixed real number.  Then for all large $T$, 
$$ 
\frac{1}{T} \text{meas}\Big\{ T\le t\le 2T: \ \ \log |\zeta(\tfrac 12+it)| \ge V \sqrt{\tfrac 12 \log \log T} \Big\}  
\sim \frac{1}{\sqrt{2\pi}} \int_{V}^{\infty} e^{-u^2/2} du. 
$$ 
\end{theorem} 
We outline the steps of the proof.  The first step is to show that $\log |\zeta(\tfrac 12+it)|$ is usually close to $\log 
|\zeta(\sigma+it)|$ for suitable $\sigma$ near $\tfrac 12$.   
 
 \begin{proposition} 
\label{Prop1}  Let $T$ be large, and suppose $T\le t\le 2T$.  Then for any $\sigma >1/2$ we have 
$$ 
\int_{t-1}^{t+1} \Big| \log |\zeta(\tfrac 12+iy)| - \log |\zeta(\sigma+iy)| \Big| dy \ll (\sigma -\tfrac 12) \log T. 
$$ 
\end{proposition}  
The proof of Proposition \ref{Prop1} is the only place where we will briefly need to mention the zeros of $\zeta(s)$. 
From now on, we set $\sigma_0 = \tfrac 12+ \frac{W}{\log T}$, for a suitable parameter $W\ge 3$ to be chosen later.  
From Proposition \ref{Prop1}, and in view of the Theorem \ref{mainthm} that we set out to prove, 
we see that if $W = o(\sqrt{\log \log T})$ then we may typically approximate 
$\log |\zeta(\tfrac 12+it)|$ by $\log |\zeta(\sigma_0+it)|$. Thus we may from now on focus on the distribution of $\log |\zeta(\sigma_0+it)|$. 

There is considerable latitude in choosing parameters such as $W$, but to fix ideas we select 
\begin{equation} 
\label{1.0} 
W= (\log \log \log T)^4, \ \  X = T^{1/(\log \log \log T)^2}, \\ \text{ and } Y = T^{1/(\log \log T)^2}. 
\end{equation} 
Here $X$ and $Y$ are two parameters that will appear shortly.
Put 
\begin{equation} 
\label{1.1} 
{\mathcal P}(s) = {\mathcal P}(s;X) = \sum_{2\le n\le X} \frac{\Lambda(n)}{n^s \log n}. 
\end{equation} 
By computing moments, it is not hard to determine the distribution of ${\mathcal P}(s)$.  

\begin{proposition}  \label{Prop2}   As $t$ varies in $T\le t\le 2T$, the distribution of $\text{Re} ({\mathcal P}(\sigma_0+it))$ is approximately 
normal with mean $0$ and variance $\sim \frac 12 \log \log T$.  
\end{proposition}

Our goal is now to connect Re(${\mathcal P}(\sigma_0+it)$) with $\log |\zeta(\sigma_0+it)|$ for most values of $t$.  
This is done in two stages.  First we introduce a Dirichlet polynomial $M(s)$ which we show is typically close to 
$\exp(-{\mathcal P}(s))$.  Define $a(n) =1$ if $n$ is composed of only primes below $X$, and it has at most $100 \log \log T$ primes 
below $Y$, and at most $100\log \log \log T$ primes between $Y$ and $X$; set $a(n)=0$ in all other cases.   Put 
\begin{equation} 
\label{1.2} 
M(s) = \sum_{n} \frac{\mu(n) a(n)}{n^s}. 
\end{equation} 
Note that $a(n)=0$ unless $n\le Y^{100\log \log T} X^{100\log \log \log T} <T^{\epsilon}$, and so $M(s)$ is a short Dirichlet 
polynomial.

\begin{proposition}  \label{Prop3} With notations as above, we have for $T\le t\le 2T$ 
$$
M(\sigma_0+it) = (1+o(1))  \exp(-\mathcal{P}(\sigma_0+it)),  
$$  
except perhaps on a subset of measure $o(T)$.  
\end{proposition} 

The final step of the proof shows that $\zeta(\sigma_0+it) M(\sigma_0+it)$ is typically close to $1$.  

\begin{proposition} \label{Prop4} 
With notations as above, 
$$ 
\int_{T}^{2T} |1-\zeta(\sigma_0+it) M(\sigma_0+it)|^2  dt = o(1), 
$$ 
so that for $T\le t\le 2T$ we have 
$$ 
\zeta(\sigma_0+it) M(\sigma_0+it) = 1+o (1), 
$$ 
except perhaps on a set of measure $o(T)$.    
\end{proposition} 

\begin{proof}[Proof of Theorem \ref{mainthm}]  To recapitulate the argument, Proposition \ref{Prop4} shows that typically $\zeta(\sigma_0+it) \approx M(\sigma_0+it)^{-1}$, 
which by Proposition \ref{Prop3} is $\approx \exp({\mathcal P}(\sigma_0+it))$, and therefore by Proposition \ref{Prop2} we 
may conclude that $\log |\zeta(\sigma_0+it)|$ is normally distributed.  Finally by Proposition \ref{Prop1} we deduce from this 
the normal distribution of $\log |\zeta(\tfrac 12+it)|$.   This completes the proof of Theorem \ref{mainthm}. 
\end{proof}

After developing the proofs of the propositions, in Section 7 we compare and contrast our approach with previous 
proofs, and also discuss possible extensions of this technique.  

\section{Proof of Proposition \ref{Prop1}} 

\noindent Put $G(s) =s(s-1)\pi^{-s/2} \Gamma(s/2)$ and $\xi(s)= G(s)\zeta(s)$ denote the completed $\zeta$-function.  
If $t$ is large and $t-1\le y\le t+1$, then by Stirling's formula $|\log G(\sigma+iy)/G(1/2+iy)| \ll (\sigma -1/2)\log t$, and 
so it is enough to prove that 
$$ 
\int_{t-1}^{t+1} \Big| \log \Big|\frac{\xi(\tfrac 12+iy)}{\xi(\sigma+iy)} \Big|\Big| dy \ll (\sigma -\tfrac 12) \log T. 
$$ 

Recall Hadamard's factorization formula 
$$ 
\xi(s) = e^{A+Bs} \prod_{\rho} \Big(1-\frac s{\rho}\Big) e^{s/\rho}, 
$$ 
where $A$ and $B$ are constants with $B= -\sum_{\rho} \text{Re }(1/\rho)$.  
Thus (assuming that $y$ is not the ordinate of a zero of $\zeta(s)$) 
$$ 
\log \Big| \frac{\xi(\tfrac 12+iy)}{\xi(\sigma+iy)}\Big| = \sum_{\rho} \log \Big| \frac{\tfrac 12 +iy-\rho}{\sigma+iy -\rho}\Big|.
$$ 
  Integrating the above over $y \in (t-1,t+1)$ we 
get 
\begin{equation} 
\label{2.1}
\int_{t-1}^{t+1} \Big| \log \Big|\frac{\xi(\tfrac 12+iy)}{\xi(\sigma+iy)}\Big| \Big| dy 
\le \sum_{\rho} \int_{t-1}^{t+1} \Big| \log \Big| \frac{\tfrac 12+iy-\rho}{\sigma+iy-\rho}\Big| \Big|dy. 
\end{equation} 

Suppose $\rho=\beta+i\gamma$ is a zero of $\zeta(s)$.  If $|t-\gamma| \ge 2$ then we check readily that 
$$ 
\int_{t-1}^{t+1} \Big| \log \Big| \frac{\tfrac 12+iy-\rho}{\sigma+iy-\rho}\Big| \Big| dy \ll \frac{(\sigma- \tfrac 12)}{(t-\gamma)^2}. 
$$ 
In the range $|t-\gamma|\le 2$ we use 
$$ 
  \int_{t-1}^{t+1} \Big| \log  \Big|\frac{\tfrac 12+iy-\rho}{\sigma+iy-\rho}\Big|\Big| dy \le \frac{1}{2} \int_{-\infty}^{\infty} \Big| \log \frac{(\beta-\tfrac 12)^2+x^2}{(\beta-\sigma)^2+x^2}\Big| dx = \pi (\sigma-\tfrac 12).
$$ 
Thus in either case 
$$ 
\int_{t-1}^{t+1} \Big| \log \Big| \frac{\tfrac 12+iy-\rho}{\sigma+iy-\rho}\Big| \Big| dy \ll \frac{(\sigma -\tfrac 12)}{1+(t-\gamma)^2}.
$$ 
Inserting this in \eqref{2.1}, and noting that there are $\ll \log (t +k)$ zeros with $k \le |t-\gamma| <k+1$,  the proposition follows.
  
\section{Proof of Proposition \ref{Prop2}} 

\noindent We begin by showing that we may restrict the sum in ${\mathcal P}(s)$ just to primes.  The contribution of cubes and higher powers of primes is clearly $O(1)$, and we need only discard the contribution of squares of primes.  By integrating out, it is easy to see that 
$$ 
\int_T^{2T} \Big| \sum_{p\le \sqrt{X}} \frac{1}{2p^{2(\sigma_0+it)} } \Big|^2 dt \ll \sum_{p_1, p_2 \le \sqrt{X}} \min\Big(T, \frac{1}{|\log (p_1/p_2)|} \Big) \ll T. 
$$ 
Therefore, the measure of the set $t\in [T,2T]$ with the contribution of prime squares being larger than $L$ (say) is at most 
$\ll T/L^2$.  In view of this, to establish Proposition \ref{Prop2}, it is enough  to prove that
$$
\mathcal{P}_0(\sigma_0 + it) := \text{Re } \sum_{p \leq X} \frac{1}{p^{\sigma_0 + it}} 
$$
has an approximately Gaussian distribution with mean 0 and variance $\sim {\tfrac 12 \log\log T}$.   We establish this by computing moments.  

\begin{lemma} \label{lemma1} Suppose that $k$ and $\ell$ are non-negative integers with $X^{k+\ell} \le {T}$.   Then, if $k\neq \ell$, 
$$ 
\int_{T}^{2T} {\mathcal P}_0(\sigma_0+it)^{k} {\mathcal P}_0(\sigma_0-it)^{\ell} dt \ll T,  
$$  
while if $k=\ell$ we have 
$$ 
\int_{T}^{2T} |{\mathcal P}_0(\sigma_0+it)|^{2k} dt = k! T (\log \log T)^{k} +O_k(T(\log \log T)^{k-1+\epsilon}). 
$$ 
\end{lemma} 
\begin{proof}  Write ${\mathcal P}(s)^k = \sum_{n} a_k(n)n^{-s}$, where $a_k(n) =0$ unless $n$ has the prime factorization  $n=p_1^{\alpha_1}\cdots p_{r}^{\alpha_r}$ where $p_1$, $\ldots$, $p_r$ are distinct primes below $X$, and $\alpha_1+\ldots+\alpha_r=k$, in which case $a_k(n) = k!/(\alpha_1! \cdots \alpha_r!)$.   Therefore, expanding out the integral, we obtain 
$$ 
\int_T^{2T} {\mathcal P}_0(\sigma_0+it)^k {\mathcal P}_0(\sigma_0-it)^{\ell} dt = T \sum_{n} \frac{a_k(n)a_\ell(n)}{n^{2\sigma_0} } 
+ O\Big( \sum_{m\neq n} \frac{a_k(m)a_{\ell}(n)}{(mn)^{\sigma_0}} \frac{1}{|\log (m/n)|} \Big). 
$$ 
If $m\neq n$, then $|\log (m/n)| \ge 1/\sqrt{mn}$ and so the off-diagonal terms above contribute $\ll \sum_{m \neq n} a_k(m)a_{\ell}(n) 
\ll X^{k+\ell}$.  Note that if $k\neq \ell$ then $a_k(n)a_\ell(n)$ is always zero, and the first statement of the lemma follows.  

It remains in the case $k=\ell$ to discuss the diagonal term $\sum_{n} a_k(n)^2/n^{2\sigma_0}$.  The terms with $n$ not being square-free are easily seen to contribute $O_k((\log \log T)^{k-2})$.  Finally the square-free terms $n$ give 
$$ 
k! \sum_{\substack{ p_1,\ldots, p_k \le X \\ p_j \text{ distinct }} } \frac{1}{(p_1 \cdots p_k)^{2\sigma_0}  } = k! \Big(\sum_{p\le X} \frac{1}{p^{2\sigma_0}} \Big)^k + O_k((\log \log T)^{k-1}), 
$$ 
and the lemma follows.
\end{proof} 

From Lemma \ref{lemma1} we see that if $X^{k} \le T$ then for odd $k$ 
$$ 
\int_T^{2T} (\text{Re } {\mathcal P}_0(\sigma_0+it)^{k} dt \ll T, 
$$ 
while if $k$ is even then 
$$ 
\frac 1T \int_{T}^{2T} (\text{Re } {\mathcal P}_0(\sigma_0+it)^k dt =  2^{-k} \binom{k}{k/2} (k/2)! (\log \log T)^{k/2} + O_k((\log \log T)^{k-1+\epsilon}). 
$$ 
These moments match the moments of a Gaussian random variable with mean zero and variance $\sim \tfrac 12\log \log T$, and since the Gaussian is determined by its moments, our proposition follows.

\section{Proof of Proposition \ref{Prop3}}  

\noindent Let us decompose ${\mathcal P}(s)$ as ${\mathcal P}_1(s)+{\mathcal P}_2(s)$, where  
$$
\mathcal{P}_1(s) = \sum_{2 \leq n \leq Y} \frac{\Lambda(n)}{n^s \log n}, \qquad  \text{and} \qquad \mathcal{P}_2(s) =
\sum_{Y < n \le X} \frac{\Lambda(n)}{n^{s} \log n}. 
$$
Put 
$$ 
{\mathcal M}_1(s) = \sum_{0\le k\le 100\log \log T} \frac{(-1)^k}{k!} {\mathcal P}_1(s)^k, \qquad 
\text{and} \qquad 
{\mathcal M}_2(s) = \sum_{0\le k\le 100\log \log \log T} \frac{(-1)^k}{k!} {\mathcal P}_2(s)^k. 
$$

\begin{lemma} \label{lemma2}  For $T\le t\le 2T$ we have 
\begin{equation} 
\label{3.1} 
|{\mathcal P}_1(\sigma_0+it)| \le \log \log T, \text{  and  }  |{\mathcal P}_2(\sigma_0+it)| \le \log \log \log T,
\end{equation} 
except perhaps for a set of measure $\ll T/\log \log \log T$.   When the bounds \eqref{3.1} hold, we have 
\begin{equation} 
\label{3.2} 
{\mathcal M}_1(\sigma_0+it) = \exp(-{\mathcal P}_1(\sigma_0+it)) \Big(1+O((\log T)^{-99})\Big),
\end{equation} 
and 
\begin{equation} 
\label{3.3} 
{\mathcal M}_2(\sigma_0+it) = \exp(-{\mathcal P}_2(\sigma_0+it)) \Big(1+O((\log \log T)^{-99})\Big).
\end{equation}
\end{lemma} 
\begin{proof}  Note that 
$$ 
\int_T^{2T} |{\mathcal P}_1(\sigma_0+it)|^2 dt \ll \sum_{2\le n_1, n_2 \le Y} \frac{\Lambda(n_1)\Lambda(n_2)}{(n_1n_2)^{\sigma_0} 
\log n_1 \log n_2} \min\Big( T, \frac{1}{|\log (n_1/n_2)|}  \Big) \ll T\log \log T,  
$$ 
and similarly 
$$ 
\int_T^{2T} |{\mathcal P}_2(\sigma_0+it)|^2 dt \ll T\log \log \log T. 
$$ 
The first assertion \eqref{3.1} follows.  
 
 If $|z|\le K$ then using Stirling's formula it is straightforward to check that 
 $$ 
 \Big| e^{z} - \sum_{0\le k\le 100K} \frac{z^k}{k!} \Big| \le e^{-99K},
 $$ 
 and therefore the estimates \eqref{3.2} and \eqref{3.3} hold.  
\end{proof}

Put $a_1(n) =1$ if $n$ is composed of at most $100\log \log T$ primes all below $Y$, and zero otherwise.  Put 
$a_2(n) =1$ if $n$ is composed of at most $100\log \log \log T$ primes all between $Y$ and $X$, and zero otherwise.  Then $M(s)=M_1(s)M_2(s)$ with 
$$ 
M_1(s) = \sum_n \frac{\mu(n)a_1(n)}{n^s} \qquad \text{and} \qquad
M_2(s) = \sum_n \frac{\mu(n) a_2(n)}{n^s}. 
$$  

\begin{lemma} \label{lemma3}  With notations as above, we have 
$$ 
\int_T^{2T} |{\mathcal M}_1(\sigma_0+it)  - M_1(\sigma_0+it)|^2 dt \ll  T(\log T)^{-60}, 
$$ 
and 
$$ 
\int_T^{2T} |{\mathcal M}_2(\sigma_0+it) - M_2(\sigma_0+it)|^2 dt \ll  T(\log \log T)^{-60}. 
$$ 
\end{lemma}  
\begin{proof}  We establish the first estimate, and the second follows similarly.  If we expand ${\mathcal M}_1(s)$ into a Dirichlet series $\sum_{n} b(n)n^{-s}$, then we may see that $|b(n)| \le 1$ for all $n$, $b(n) =0$ unless $n 
\le Y^{100\log \log T}$ is composed only of primes below $Y$, and $b(n) = \mu(n) a_1(n)$ if  $\Omega(n)\le 100\log \log T$.  (It is the presence of prime powers in ${\mathcal P}(s)$ that prevents ${\mathcal M}_1(s)$ from simply being $M_1(s)$.)  Thus, putting $c(n) = b(n) - \mu(n)a_1(n)$ temporarily, we see that 
$$ 
\int_T^{2T} |{\mathcal M}_1(\sigma_0+it) - M_1(\sigma_0 +it)|^2 dt \ll 
\sum_{n_1, n_2} \frac{|c(n_1)c(n_2)|}{(n_1n_2)^{\sigma_0}} \min \Big( T ,\frac{1}{|\log (n_1/n_2)|}\Big). 
$$ 
The terms with $n_1 \neq n_2$ contribute (since $|\log (n_1/n_2)| \gg 1/\sqrt{n_1n_2}$ in that case) 
$$ 
\ll \sum_{n_1 \neq n_2 \le Y^{100\log \log T}} 1 \ll T^{\epsilon}.  
$$ 
The diagonal terms $n_1=n_2$ contribute, for any $1<r<2$, 
$$ 
\ll T \sum_{\substack{ p| n \implies p\le Y \\ \Omega(n) >100\log \log T}} \frac{1}{n} \ll T r^{-100\log \log T} \prod_{p\le Y}\Big(1 +\frac{r}{p}+ \frac{r^2}{p^2}+\ldots \Big).  
$$ 
Choosing $r=e^{2/3}$, say, the above is $\ll T(\log T)^{-60}$.  
\end{proof} 

\begin{proof}[Proof of Proposition \ref{Prop3}]   From Lemma \ref{lemma3} it follows that except on a set of measure $o(T)$, one 
has $M_1(\sigma_0+it) = {\mathcal M}_1(\sigma_0+it) + O((\log T)^{-25})$.  Moreover, from \eqref{3.2}
 (except on a set of measure $o(T)$) we note that ${\mathcal M}_1(\sigma_0+it) = \exp(-{\mathcal P}_1(\sigma_0+it))(1+O((\log T)^{-99}))$, and by \eqref{3.1} that $ (\log T)^{-1} \ll |{\mathcal M}_1(\sigma_0+it) | \ll \log T$.  Therefore, we may conclude that, 
 except on a set of measure $o(T)$, 
 $$ 
 M_1(\sigma_0+it) = {\mathcal M}_1(\sigma_0+it) + O((\log T)^{-25}) = \exp(-{\mathcal P}_1(\sigma_0+it)) (1+O((\log T)^{-20})). 
 $$ 
Similarly, except on a set of measure $o(T)$, we have 
$$ 
M_2(\sigma_0+it) = {\mathcal M}_2(\sigma_0+it) + O((\log \log T)^{-25}) = \exp(-{\mathcal P}_2(\sigma_0+it)) (1+O((\log \log T)^{-20})).
$$ 
Multiplying these estimates we obtain 
$$ 
M(\sigma_0+it) = \exp(-{\mathcal P}(\sigma_0+it)) (1+O((\log \log T)^{-20})),
$$ 
completing our proof.  \end{proof}



\section{Proof of Proposition \ref{Prop4}} 
  
\noindent For $T\le t\le 2T$, one has $\zeta(\sigma_0+it) = \sum_{n\le T} n^{-\sigma_0-it} + O(T^{-\frac 12})$, and 
so 
$$ 
\int_T^{2T} \zeta(\sigma_0 +it) M(\sigma_0+it) dt = \sum_{n\le T} \sum_{m} \frac{a(m)\mu(m)}{(mn)^{\sigma} } 
\int_T^{2T} (mn)^{-it} dt + O(T^{\frac 12+\epsilon}) = T + O(T^{\frac 12+\epsilon}). 
$$ 
Therefore, expanding the square, we see that 
\begin{equation} 
\label{5.1} 
\int_{T}^{2T} |1-\zeta(\sigma_0 + it) M(\sigma_0 + it)|^2   dt =  
 \int_T^{2T} |\zeta(\sigma_0+it) M(\sigma_0+it)|^2 dt -T +O(T^{\frac 12 +\epsilon}). 
\end{equation}  
It remains to evaluate the integral above, and to do this we shall use the following familiar lemma (see for example Lemma 6 of Selberg \cite{Selberg2}).  For completeness 
we include a quick proof of the lemma in the next section, and we note that we give only a version sufficient for our purposes and not 
the sharpest known result.   

 \begin{lemma} \label{lemma4}  Let $h$ and $k$ be non-negative integers, with $h,k \le T$.  Then, for any $1\ge \sigma > \tfrac 12$, 
 \begin{align*}
\int_{T}^{2T}  \Big ( \frac{h}{k} \Big )^{it} |\zeta(\sigma + it )|^2 dt 
 &= \int_T^{2T} \Big( \zeta(2\sigma) \Big( \frac{(h,k)^2}{hk}\Big)^{\sigma}  + \Big(\frac{t}{2\pi}\Big)^{1-2\sigma} \zeta(2-2\sigma) \Big( \frac{(h,k)^2}{hk}\Big)^{1-\sigma}\Big) dt\\ 
 & + O(T^{1-\sigma+\epsilon} \min(h,k)).   
\end{align*}
\end{lemma}  

Assuming Lemma \ref{lemma4}, we now complete the proof of Proposition \ref{Prop4}.   In view of \eqref{5.1} 
we must show that 
\begin{equation} 
\label{5.2} 
\sum_{h,k} \frac{\mu(h)\mu(k) a(h)a(k)}{(hk)^{\sigma_0}} \int_T^{2T} \Big(\frac{h}{k} \Big)^{it} |\zeta(\sigma_0+it)|^2 \sim T, 
\end{equation} 
and to do this we appeal to Lemma \ref{lemma4}.  The error terms are easily seen to be $o(T)$, and we now focus on the 
main terms arising from Lemma \ref{lemma4}, beginning with the first main term.  This contributes 
\begin{equation} 
\label{5.3} 
T \zeta(2\sigma_0) \sum_{h,k} \frac{\mu(h)\mu(k)a(h)a(k)}{(hk)^{2\sigma_0}} (h,k)^{2\sigma_0}. 
\end{equation} 
Write $h=h_1h_2$ where $h_1$ is composed only of primes below $Y$, and $h_2$ of primes between $Y$ and $X$, and then 
$a(h) =a_1(h_1) a_2(h_2)$ in the notation of section 3.  Writing similarly $a(k)=a_1(k_1)a_2(k_2)$, we see that the quantity in 
\eqref{5.3} factors as 
\begin{equation} 
\label{5.4} 
T \zeta(2\sigma_0) \Big(\sum_{h_1,k_1} \frac{\mu(h_1)\mu(k_1)a_1(h_1)a_1(k_1)}{(h_1k_1)^{2\sigma_0}} (h_1,k_1)^{2\sigma_0}\Big) 
\Big( \sum_{h_2,k_2} \frac{\mu(h_2)\mu(k_2)a_2(h_2)a_2(k_2)}{(h_2k_2)^{2\sigma_0}} (h_2,k_2)^{2\sigma_0}\Big). 
\end{equation}  

Consider the first factor in \eqref{5.4}.  If we ignore the condition that $h_1$ and $k_1$ must have at most $100 \log \log T$ 
prime factors, then the resulting sum is simply 
$$ 
\sum_{\substack{ h_1, k_1 \\ p|h_1k_1\implies p\le Y}} \frac{\mu(h_1)\mu(k_1)}{(h_1k_1)^{2\sigma_0}} (h_1,k_1)^{2\sigma_0}
= \prod_{p\le Y} \Big(1- \frac{1}{p^{2\sigma_0}}\Big). 
$$ 
In approximating the first factor by the product above, we incur an error term which is at most (by symmetry we may suppose that 
$h_1$ has many prime factors) 
\begin{align*}
&\ll \sum_{\substack{ h_1, k_1 \\ p|h_1k_1\implies p\le Y \\ \Omega(h_1) >100\log \log T}} \frac{|\mu(h_1)\mu(k_1)|}{(h_1k_1)^{2\sigma_0}} (h_1,k_1)^{2\sigma_0}\\
& \ll  e^{-100\log \log T}  \sum_{\substack{ h_1, k_1 \\ p|h_1k_1\implies p\le Y }} \frac{|\mu(h_1)\mu(k_1)|}{(h_1k_1)^{2\sigma_0}} (h_1,k_1)^{2\sigma_0} e^{\Omega(h_1)}\\
& \ll 
(\log T)^{-100} \prod_{p\le Y} \Big( 1+ \frac{1+2e}{p} \Big) \ll   (\log T)^{-90} .
\end{align*} 
Similarly one obtains that the second factor in \eqref{5.4} is 
$$ 
\prod_{Y<p\le X} \Big(1-\frac{1}{p^{2\sigma_0}}  \Big) + O((\log \log T)^{-90}). 
$$ 
Using these in \eqref{5.4}, we obtain that the first main term is 
$$ 
\sim T\zeta(2\sigma_0) \prod_{p\le X} \Big(1-\frac{1}{p^{2\sigma_0}}\Big) = T \prod_{p> X} \Big(1-\frac{1}{p^{2\sigma_0}}\Big)^{-1} 
\sim T
$$ 
since $(\sigma_0 - \tfrac 12)^{-1} = o(\log T / \log X)$. 
In the same way we see that the second main term arising from Lemma \ref{lemma4} is 
\begin{align*}
&\zeta(2-2\sigma_0) \Big(\int_T^{2T} \Big(\frac{t}{2\pi}\Big)^{1-2\sigma_0} dt \Big)  \sum_{h,k} \frac{\mu(h)\mu(k)a(h)a(k)}{hk} (h,k)^{2-2\sigma_0} \\
& \sim  \Big(\int_T^{2T} \Big(\frac{t}{2\pi}\Big)^{1-2\sigma_0} dt \Big)  \zeta(2-2\sigma_0) 
\prod_{p\le X} \Big(1- \frac{2}{p} + \frac{1}{p^{2\sigma_0}}\Big) = o(T).
\end{align*}
This completes our proof of \eqref{5.2}, and hence of Proposition \ref{Prop4}.    
   
 \section{Proof of Lemma \ref{lemma4}} 
 
 \noindent Put $G(s) = \pi^{-s/2} s(s-1) \Gamma(s/2)$, so that $\xi(s) = G(s) \zeta(s)=\xi(1-s)$ is the completed zeta function.   
 Define for any given $s\in {\Bbb C}$ 
 $$ 
 I(s) = I(\overline{s}) = \frac{1}{2\pi i} \int_{(c)} \xi(z+s) \xi(z+\overline{s}) e^{z^2} \frac{dz}{z}, 
 $$ 
 where the integral is over the line from $c-i\infty$ to $c+i\infty$ for any $c>0$.   By moving the 
 line of integration to the left, and using the functional equation $\xi(z+s)\xi(z+\overline{s}) = \xi(-z+(1-s))\xi(-z+(1-\overline{s}))$ we obtain that 
 \begin{equation} 
 \label{6.1} 
 |\zeta(s)|^2 = \frac{1}{|G(s)|^2} \Big( I(s) + I(1-s) \Big). 
 \end{equation} 
   
From now on suppose that $s=\sigma+it$ with $T\le t\le 2T$, and $1 \ge \sigma  \ge \tfrac 12$.   If $z$ is a complex number with real part $c=1-\sigma + 1/\log T$, then an application of Stirling's formula gives 
$$ 
\frac{G(z+s)G(z+\overline{s})}{|G(s)|^2} =\Big( \frac{t}{2\pi }\Big)^{z} \Big( 1+ O\Big( \frac{|z|}{T}\Big)\Big) .
$$ 
Therefore, we see that 
$$ 
\frac{I(s)}{|G(s)|^2} = \frac{1}{2\pi i}\int_{(1-\sigma+1/\log T)} \Big(\frac{t}{2\pi}\Big)^{z} \zeta(z+s)\zeta(z+\overline{s}) 
e^{z^2}\frac{dz}{z} +  O(T^{-\sigma +\epsilon}). 
$$  
Since we are in the region of absolute convergence of $\zeta(z+s)$ and $\zeta(z+\overline{s})$, we obtain 
\begin{align}
\label{6.2}
\int_T^{2T} \Big(\frac{h}{k} \Big)^{it} \frac{I(s)}{|G(s)|^2}dt 
&= \frac{1}{2\pi i} \int_{(1-\sigma+1/\log T)} \frac{e^{z^2}}{ z} \sum_{m,n=1}^{\infty} 
\frac{1}{(mn)^{z+\sigma}} \Big( \int_T^{2T} \Big(\frac{hm}{kn}\Big)^{it} \Big(\frac{t}{2\pi}\Big)^z dt \Big) dz \nonumber\\
&\hskip .5 in + O(T^{1-\sigma+\epsilon}).  
\end{align}
In the integral in \eqref{6.2}, we distinguish the diagonal terms $hm=kn$ from the off-diagonal terms $hm\neq kn$.  The 
diagonal terms $hm=kn$ may be parametrized as $m=Nk/(h,k)$ and $n=Nh/(h,k)$, and therefore these terms contribute 
\begin{equation} 
\label{6.3} 
\frac{1}{2\pi i} \int_{(1-\sigma+1/\log T)} \frac{e^{z^2}}{ z} \zeta(2z+2\sigma) \Big( \frac{(h,k)^2}{hk}\Big)^{z+\sigma} 
\Big(\int_{T}^{2T} \Big(\frac{t}{2\pi}\Big)^z dt \Big) dz. 
\end{equation} 
As for the off-diagonal terms, the inner integral over $t$ may be bounded by $\ll T^{1-\sigma} \min(T, 1/|\log (hm/kn)|)$, and 
therefore these contribute 
\begin{equation} 
\label{6.4} 
\ll T^{1-\sigma} \sum_{\substack{m,n =1 \\ hm\neq kn}}^{\infty} \frac{1}{(mn)^{1+1/\log T}} \min \Big( T , \frac{1}{|\log (hm/kn)|}\Big) 
\ll \min(h,k) T^{1-\sigma+\epsilon}. 
\end{equation} 
The final estimate above follows by first discarding terms with $hm/(kn)>2$ or $<1/2$, and for the remaining terms (assume that $k\le h$) noting that for a given $m$ the sum over values $n$ may be bounded by $kT^{\epsilon}$ (here it may be useful to distinguish the cases $hm>T$ and $hm<T$).     
 
 From \eqref{6.2}, \eqref{6.3} and \eqref{6.4}, we conclude that 
 \begin{align} 
 \label{6.5} 
   \int_T^{2T} \Big(\frac{h}{k} \Big)^{it} \frac{I(s)}{|G(s)|^2}dt& = \frac{1}{2\pi i} \int_{(1-\sigma+1/\log T)} \frac{e^{z^2}}{ z} \zeta(2z+2\sigma) \Big( \frac{(h,k)^2}{hk}\Big)^{z+\sigma} 
\Big(\int_{T}^{2T} \Big(\frac{t}{2\pi}\Big)^z dt \Big) dz \nonumber \\
&\hskip .5 in + O(\min(h,k) T^{1-\sigma+\epsilon}).
\end{align} 
 A similar argument gives 
 \begin{align} 
 \label{6.6} 
 \int_T^{2T}\Big(\frac{h}{k} \Big)^{it} &\frac{I(1-s)}{|G(s)|^2}dt =  O(T^{1-\sigma+\epsilon} \min(h,k)) \nonumber\\ 
 &+ \frac{1}{2\pi i} \int_{(\sigma+1/\log T)} 
 \frac{e^{z^2}}{z} \zeta(2z+2-2\sigma)\Big(\frac{(h,k)^2}{hk}\Big)^{z+1-\sigma} \Big(\int_T^{2T} \Big(\frac{t}{2\pi} \Big)^{z+1-2\sigma} dt\Big) dz.
 \end{align}    
 With a suitable change of variables, we can combine the main terms in \eqref{6.5} and \eqref{6.6} as 
 $$ 
 \frac{1}{2\pi i} \int_{(1+1/\log T)} \zeta(2z) \Big(\frac{(h,k)^2}{hk}\Big)^z \Big(\int_{T}^{2T} \Big(\frac{t}{2\pi} \Big)^{z-\sigma} dt\Big) 
 \Big( \frac{e^{(z-\sigma)^2}}{z-\sigma} + \frac{e^{(z-1+\sigma)^2}}{z-1+\sigma}\Big) dz, 
 $$    
 and moving the line of integration to the left we obtain the main term of the lemma as the residues of the poles at $z=\sigma$ and $z=1-\sigma$ (note that the potential pole at $z=1/2$ from $\zeta(2z)$ is canceled by a zero of $e^{(z-\sigma)^2}/(z-\sigma) + e^{(z-1+\sigma)^2}/(z-1+\sigma)$ there).   This completes our proof of Lemma \ref{lemma4}.
 
 \section{Discussion}

 \noindent In common with Selberg's proof of Theorem \ref{mainthm} our proof relies on the Gaussian distribution of short sums over 
 primes, as in Proposition \ref{Prop2}.  In contrast with Selberg's proof, we do not need to invoke zero density estimates for $\zeta(s)$,  
 the easier mean-value theorem in Proposition \ref{Prop4} provides for us a sufficient substitute.  Selberg's original proof also 
 used an intricate argument expressing $\log \zeta(s)$ in terms of primes and zeros; an elegant alternative approach was given 
 by Bombieri and Hejhal \cite{BomHej}, although they too require a strong zero density result near the critical line.  We should also point out that by just focussing on the central limit theorem, we have not obtained asymptotic formulae for the moments of $\log |\zeta(\tfrac 12+it)|$ which Selberg established.

  In Selberg's approach, it was easier to handle Im $\log \zeta(\tfrac 12+it)$, and the case of $\log |\zeta(\tfrac 12+it)|$ entailed additional technicalities.  In contrast, our method works well for $\log |\zeta(\tfrac 12+it)|$ but requires substantial modifications to handle Im $\log \zeta(\tfrac 12+it)$.  The reason is that Proposition \ref{Prop4} guarantees that typically $|\zeta(\sigma_0+it)| \approx |M(\sigma_0+it)|^{-1}$, but it could be that Im $ \log \zeta(\frac 12+it)$ 
 and Im $ \log M(\sigma_0+it)^{-1}$ are not typically close but differ by a substantial integer multiple of $2\pi$.  
In this respect our argument shares some similarities with Laurin\v{c}ikas's proof of Selberg's central limit theorem \cite{Laur}, 
 which relies on bounding small moments of $|\zeta(\tfrac 12+it)|$ using Heath-Brown's work on fractional moments \cite{HBFrac1}. In particular, Laurin\v{c}ikas's argument also breaks down for the imaginary part of $\log \zeta(\tfrac 12 + it)$.

 
We can quantify the argument given here, providing a rate of convergence to the limiting distribution.  
 With more effort (in particular taking higher moments in Lemma \ref{lemma2} to obtain better bounds on the exceptional set there) 
 we can recover previous results in this direction, but have not been able to obtain anything stronger.  We also remark that the argument also gives the joint distribution of $\log |\zeta(\tfrac 12+it)|$ and $\log |\zeta(\tfrac 12+it +i\alpha)|$ (for any fixed non-zero $\alpha \in {\Bbb R}$) and shows that these are distributed like independent Gaussians.  One can allow for more than one shift, and also keep track of the uniformity in $\alpha$.   
 

Our proof of Proposition \ref{Prop3} (in Section 4) involved splitting the mollifier $M(s)$ into two factors, or equivalently of the prime sum ${\mathcal P}(s)$ into two pieces.  We would have liked to get away with just one prime sum, but this barely fails.  In order to use Proposition \ref{Prop1} successfully, we are forced to take $W= o(\sqrt{\log \log T})$.   To mollify successfully on the 
$\frac 12+\frac{W}{\log T}$ line (see Proposition \ref{Prop4}) we need to work with primes going up to roughly $T^{\frac 1W}$.  If $W=o(\sqrt{\log \log T})$ then this length is $T^{A/\sqrt{\log \log T}}$ for a large parameter $A$, and if we try to expand $\exp({\mathcal P}(s))$ into a series (as in Section 4) we will be forced to take more than $\sqrt{\log \log T} $ terms in the exponential series.  This 
leads to Dirichlet polynomials that are just a little too long.  We resolve this (see Section 4) by splitting ${\mathcal P}$ into two terms, exploiting the fact that the longer sum ${\mathcal P}_2$ has a significantly smaller variance.

 Propositions \ref{Prop1}, \ref{Prop2}, and \ref{Prop3} in our argument are quite general and analogs may be established for 
 higher degree $L$-functions in the $t$-aspect.  An analog of Proposition \ref{Prop4} however can at present only be established for $L$-functions of degree $2$ (relying here upon information on the shifted convolution problem), and unknown for degrees $3$ or higher.   However, some hybrid results are possible.   For example, by adapting the techniques in \cite{CIS1, CIS2} we can establish an analog of Proposition \ref{Prop4} for twists of a fixed $GL(3)$ $L$-function by primitive Dirichlet characters with conductor below $Q$.  In this way one can show that as $\chi$ ranges over all primitive Dirichlet characters with conductor below $Q$, and $t$ ranges between $-1$ and $1$, the distribution of $\log |L(\tfrac 12+it, f\times \chi)|$ is approximately normal with mean $0$ and variance $\sim \tfrac 12\log \log Q$; here $f$ is a fixed eigenform on $GL(3)$.

 Keating and Snaith \cite{KeSn1} have conjectured that central values of $L$-functions in families have a log normal distribution with an appropriate mean and  variance depending on the family.   For example, we may consider the family of quadratic Dirichlet $L$-functions $L(\tfrac 12,\chi_d)$ where $d$ ranges over fundamental discriminants of size $X$.  In this setting, we may carry out the arguments of Propositions \ref{Prop2}, \ref{Prop3} and \ref{Prop4} and conclude that $\log L(\sigma_0,\chi_d)$ has a normal distribution with mean $\sim \tfrac 12\log \log X$ and variance $\sim \log \log X$, provided that $\sigma_0 =\tfrac 12+ \frac{W}{\log X}$ where $W$ is any function with $W\to \infty$ as $X\to \infty$ and with $\log W = o(\log \log X)$.   However in this situation we do not have an analog of Proposition \ref{Prop1} allowing us to pass from this to the central value; indeed, our knowledge at present does not exclude the possibility that $L(\tfrac 12,\chi_d)=0$ for a positive proportion of discriminants $d$.  
 
 Finally we remark that the proof presented here was suggested by earlier work of the authors \cite{RS1}, where general one sided central limit theorems towards the Keating-Snaith conjectures are established.  
 
 \bibliography{Refs}{}

\begin{thebibliography}{1}

\bibitem{BomHej}
E.~Bombieri and D.~A. Hejhal.
\newblock On the distribution of zeros of linear combinations of {E}uler
  products.
\newblock {\em Duke Math. J.}, 80(3):821--862, 1995.

\bibitem{CIS2}
J.~B. Conrey, H.~Iwaniec, and K.~Soundararajan.
\newblock The sixth power moment of {D}irichlet {$L$}-functions.
\newblock {\em Geom. Funct. Anal.}, 22(5):1257--1288, 2012.

\bibitem{CIS1}
J.~Brian Conrey, Henryk Iwaniec, and Kannan Soundararajan.
\newblock Critical zeros of {D}irichlet {$L$}-functions.
\newblock {\em J. Reine Angew. Math.}, 681:175--198, 2013.

\bibitem{HBFrac1}
D.~R. Heath-Brown.
\newblock Fractional moments of the {R}iemann zeta function.
\newblock {\em J. London Math. Soc. (2)}, 24(1):65--78, 1981.

\bibitem{KeSn1}
J.~P. Keating and N.~C. Snaith.
\newblock Random matrix theory and {$L$}-functions at {$s=1/2$}.
\newblock {\em Comm. Math. Phys.}, 214(1):91--110, 2000.

\bibitem{Laur}
A.~Laurinchikas.
\newblock A limit theorem for the {R}iemann zeta-function on the critical line.
  {II}.
\newblock {\em Litovsk. Mat. Sb.}, 27(3):489--500, 1987.

\bibitem{RS1}
Maksym Radziwi{\l }{\l} and K.~Soundararajan.
\newblock Moments and distribution of central {L}-values of quadratic twists of
  elliptic curves.
\newblock {\em Invent. Math.}, pages 1--40, 2015.

\bibitem{Selberg2}
Atle Selberg.
\newblock Contributions to the theory of the {R}iemann zeta-function.
\newblock {\em Arch. Math. Naturvid.}, 48(5):89--155, 1946.

\bibitem{Selberg1}
Atle Selberg.
\newblock Old and new conjectures and results about a class of {D}irichlet
  series.
\newblock In {\em Proceedings of the {A}malfi {C}onference on {A}nalytic
  {N}umber {T}heory ({M}aiori, 1989)}, pages 367--385. Univ. Salerno, Salerno,
  1992.

\end{thebibliography}
 \bibliographystyle{plain}  
\end{document}